\documentclass[11pt]{article}
\usepackage{amsmath,amssymb,amsfonts,amsthm}
\usepackage{epsfig}
\usepackage[margin=1.25 in]{geometry}
\allowdisplaybreaks

\newtheorem{theorem}{Theorem}[section]
\newtheorem{lemma}[theorem]{Lemma}

\theoremstyle{definition}
\newtheorem{definition}[theorem]{Definition}

\theoremstyle{remark}

\numberwithin{equation}{section}
\newcommand{\nn}{\nonumber}
\def\refe#1{(\ref{#1})}

\def\d{\,{\rm d}}

\begin{document}

\title{\bf Global well-posedness
of the time-dependent 
Ginzburg--Landau superconductivity model
in curved polyhedra 
\footnote{This work was supported in part by 
the NSFC under Grant No. 11301262 and 11401587. }
}
 
\author{Buyang Li \footnote{Department of Mathematics, 
Nanjing University, Nanjing, 210093, P.R. China.
{\tt buyangli@nju.edu.cn}}  
\quad
and\quad Chaoxia Yang
\footnote {College of Science, 
Nanjing University of 
Posts and Telecommunications,
Nanjing, 210023, 
P.R. China. }
}

\date{}

\maketitle

\begin{abstract}
We study the time-dependent 
Ginzburg--Landau equations in 
a three-dimensional curved polyhedron 
(possibly nonconvex). 
Compared with the previous works,
we prove existence and uniqueness
of a global weak solution based on weaker regularity
of the solution in the presence of edges or corners,
where the magnetic potential may not be in 
$L^2(0,T;H^1(\Omega)^3)$.

\end{abstract}

\maketitle

\section{Introduction}
\setcounter{equation}{0}

The Ginzburg--Landau theory has been widely
accepted to describe macroscopic superconductivity phenomena in
superconductors.
Based on the Ginzburg--Landau theory \cite{GL}
and its extension to the time-dependent model \cite{GE},
the macroscopic state of a superconductor  
can be described by a complex-valued order parameter $\psi$
whose modulus represents the superconductivity density,
and a real vector-valued magnetic potential ${\bf A}$
whose curl represents the magnetic induction field.
If a superconductor occupies a domain $\Omega$, then
$\psi$ and ${\bf A}$ are governed by 
the time-dependent Ginzburg--Landau equations (TDGL) 
which, in a nondimensionalization form, 
can be written as
\begin{align}
&\eta\frac{\partial \psi}{\partial t}
+ \bigg(\frac{i}{\kappa} \nabla + \mathbf{A}\bigg)^{2} \psi
 + (|\psi|^{2}-1) \psi + i\eta\kappa\psi\phi = 0,
\label{GLLPDEq1}\\[5pt]
&\frac{\partial \mathbf{A}}{\partial t} 
+ \nabla\times(\nabla\times{\bf A})
+\nabla \phi+  {\rm Re}\bigg[\psi^*\bigg(\frac{i}{\kappa} \nabla 
+ \mathbf{A}\bigg) \psi\bigg] =  \nabla\times {\bf H},
\label{GLLPDEq2}
\end{align} 
where $\eta$ and $\kappa$ are physical constants
and ${\bf H}$ is the external magnetic field,
$\psi^*$ denotes the complex conjugate of $\psi$, 
and the electric potential $\phi$ is 
an unknown real scalar-valued function. 
Physically there should be $0\leq |\psi|\leq 1$, 
where $|\psi|=0$ and $|\psi|=1$ correspond to 
the normal state and superconducting state,  respectively,
and $0<|\psi|<1$ represents a mixed state.
More detailed discussion of the physics can be found in \cite{CHO,DGP92,Gennes,Tinkham}.

Clearly, the three unknown variables $\psi$, ${\bf A}$ and $\phi$
cannot be determined uniquely
by the two equations above.
Thus an additional gauge condition is required.
The zero electric potential gauge $\phi=0$ is 
widely used by physicists  
\cite{FUD91,GKLLP96, LMG91,WA02}.
Under this gauge, existence and uniqueness of solutions  
have been proved in \cite{Du94}
when the superconductor occupies a smooth domain.
Finite element approximations of the equations 
and its convergence have also been studied in many
works, e.g., \cite{Du05,Mu97,MH98}.
Existence and uniqueness of weak solutions
have also been proved in \cite{Tang95}
under the Coulomb gauge $\nabla\cdot{\bf A}=0$
(also in smooth domains),
which is less used in numerical simulations as it increases
the computational cost. To avoid the degeneracy caused by
the zero electric potential gauge,
we focus on the Lorentz gauge $\phi=-\nabla\cdot{\bf A}$ 
in this paper. Under this gauge, 
\refe{GLLPDEq1}-\refe{GLLPDEq2} 
reduce to 
\begin{align}
&\eta\frac{\partial \psi}{\partial t} 
+ \bigg(\frac{i}{\kappa} \nabla + \mathbf{A}\bigg)^{2} \psi
 + (|\psi|^{2}-1) \psi 
 -i\eta \kappa \psi \nabla\cdot{\bf A} = 0,
\label{PDE1}\\[5pt]
&\frac{\partial \mathbf{A}}{\partial t} 
+ \nabla\times(\nabla\times{\bf A})
-\nabla(\nabla\cdot{\bf A}) 
+  {\rm Re}\bigg[\psi^*\bigg(\frac{i}{\kappa} \nabla 
+ \mathbf{A}\bigg) \psi\bigg] =  \nabla\times {\bf H} .
\label{PDE2}
\end{align}
The natural boundary and initial 
conditions for the problem are 
\begin{align}
& \nabla \psi\cdot{\bf n}= 0,
\quad {\bf A}\cdot{\bf n}=0, 
\quad (\nabla\times{\bf A})\times{\bf n}= {\bf H} \times{\bf n} , 
\quad \mathrm{on}\,\,\,\, \partial \Omega \times (0,T),&
\label{bc}\\
& \psi(x,0) = \psi_{0}(x), \quad \mathbf{A}(x,0) =
\mathbf{A}_{0}(x), \qquad\, \mathrm{in}\,\,\,\,  \Omega \, ,
\label{init}
\end{align}
where $\mathbf{n}$ denotes the unit outward 
normal vector on the boundary
$\partial\Omega$.

Existence and uniqueness of solutions of 
\refe{PDE1}-\refe{init} have been
proved in \cite{CHL} in smooth domains,
where the authors also proved that 
the Lorentz and
temporal gauges are equivalent in 
producing the physical quantities 
such as $|\psi|$ and $\nabla\times{\bf A}$.  
The long time behavior of the solutions
and vortex dynamics have been studied in 
\cite{Spirn02}.
Finite element approximations of
\refe{PDE1}-\refe{init} and their convergence 
have also been studied in many works, e.g., 
\cite{Chen97,CD01,GLS}. 
In both the theoretical and numerical analyses, 
the proofs presented in these works have used the
equivalence relation 
$\|\nabla\times{\bf A}\|_{{\bf L}^2}^2
+\|\nabla\cdot{\bf A}\|_{L^2}^2\sim 
\|\nabla {\bf A}\|_{{\bf L}^2}^2$, which holds for any 
${\bf A}\in H^1(\Omega)^3$ such that 
${\bf A}\cdot{\bf n}=0$ on $\partial\Omega$.
The previous analyses based on this equivalence relation 
showed that
${\bf A}\in L^2(0,T;H^1(\Omega)^3)\hookrightarrow 
L^2(0,T;L^6(\Omega)^3)$.
However, when the domain contains edges or corners
whose interior angles are larger than $180^\circ$, 
the equivalence relation does not hold any more
and ${\bf A}$ does not have the regularity above
in general.  
Thus the previous analyses 
cannot be extended to nonconvex polyhedra directly. 
Since numerical simulations of the TDGL 
in nonsmooth domains are important for physicists 
to study the surface effects in superconductivity  
\cite{PSB02,VMB03}, it is important to know whether
the TDGL is well-posed and whether the finite element solutions
converge in a general curved polyhedron. 
This question is answered in \cite{LZ1,LZ2}
in the two-dimensional
case, where a 
new equivalent system of equations
was introduced 
to compute the solutions of the TDGL.
Existence and uniqueness of weak solutions
for the TDGL, as well as the new system, 
were proved in two-dimensional 
nonconvex polygons. 
It is noted that the standard finite element solution of the
TDGL may converge to an incorrect solution,
while the finite element solution 
of the new equivalent system converges to the true solution. 
Unfortunately, the new equivalent system introduced
in \cite{LZ1,LZ2} cannot be extended to the three-dimensional 
model, and the regularity of the solution in a 
three-dimensional curved polyhedron 
is even weaker than the two-dimensional solutions. 
In this paper, we prove existence of weak solutions for
the initial-boundary value problem 
\refe{PDE1}-\refe{init}
in a general three-dimensional curved polyhedron
based on the weaker embedding inequality
$\|{\bf A}\|_{{\bf L}^{3+\delta}(\Omega)}
\leq C\|{\bf A}\|_{{\bf H}_{\rm n}({\rm curl,div})}$,
by constructing approximating solutions
which preserve the physical property 
$0\leq |\psi|\leq 1$.
Then we prove uniqueness of the weak solution
based on the regularity proved.

We would like to mention that, similar as the
two-dimensional case, 
the standard finite element solutions
of \refe{PDE1}-\refe{init}
would converge to an incorrect solution
in some nonsmooth domains.
Numerical approximation of \refe{PDE1}-\refe{init}
in general three-dimensional curved polyhedra
is challenging (under either gauge).
The current paper provides a 
theoretical foundation for numerical analysis 
of the TDGL in such domains.

\section{Main results}\label{femmethod}
\setcounter{equation}{0}

Let $\Omega$ be a curved polyhedron, i.e.
a bounded Lipschitz domain with piecewise smooth boundary such
that any boundary point is contained in a neighborhood which is 
$C^\infty$-diffeomorphic to a neighborhood of a boundary point
of a polyhedron; see section 4 of \cite{BS87}
and \cite{CD99}.
For any nonnegative integer $k$, we 
denote by $H^{k}(\Omega)$ and ${\cal H}^{k}(\Omega)$
the conventional Sobolev spaces
of real-valued and complex-valued 
functions defined in $\Omega$, respectively,
with $L^2(\Omega)=H^0(\Omega)$ and
${\mathcal  L}^2(\Omega)={\mathcal  H}^0(\Omega)$;
see \cite{Adams}. 
To simplify the notations, we
denote $H^{k}=H^{k}(\Omega)$,
${\mathcal  H}^{k}={\mathcal  H}^{k}(\Omega)$, 
$L^p=L^p(\Omega)$, 
${\mathcal  L}^p={\mathcal  L}^p(\Omega)$
and ${\bf L}^p:= L^p(\Omega)^3 $
for $1\leq p\leq\infty$,
and we define the function spaces  
\begin{align*}
&{\bf H}^1_{\rm n}(\Omega):=\{{\bf a}\in H^1(\Omega)^3: 
{\bf a}\cdot{\bf n}=0~\,\mbox{on}~\,\partial\Omega\},
\quad
{\bf H}({\rm curl}) :=\{{\bf a}\in {\bf L}^2: \nabla\times{\bf a}\in L^2\} ,\\
&{\bf H}_{\rm n}({\rm div}) :=\{{\bf a}\in {\bf L}^2: \nabla\cdot{\bf a}\in L^2~\mbox{and}
~{\bf a}\cdot{\bf n}=0~\mbox{on}~\partial\Omega\} ,\\
&{\bf H}_{\rm n}({\rm curl},{\rm div}):=\{{\bf a}\in {\bf L}^2: \nabla\times {\bf a}\in {\bf L}^2,
~\nabla\cdot{\bf a}\in L^2~\mbox{and}
~{\bf a}\cdot{\bf n}=0~\mbox{on}~\partial\Omega\},
\end{align*}
where $C^\infty(\overline\Omega)$ is dense in 
${\bf H}_{\rm n}({\rm div})$
but may not be dense in ${\bf H}_{\rm n}({\rm curl},{\rm div})$
(depending on the convexity of the domain).
For any two functions $f,g\in {\mathcal  L}^2$ we define
$(f,g)=\int_\Omega f(x)g(x)^*\d x $,
and for any two vector fields ${\bf f},{\bf g}\in {\bf L}^2$ we define
$({\bf f},{\bf g})=\int_\Omega {\bf f}(x)\cdot{\bf g}(x)^*\d x $.

\begin{definition}\label{DefWSol}
{\bf(Weak solution)}
{\it 
The pair $(\psi,{\bf A})$ is called a weak solution of 
{\rm\refe{PDE1}-\refe{init}} if
\begin{align*}
&\psi\in C([0,T];{\mathcal  L}^2)\cap 
L^\infty(0,T;{\mathcal  H}^1) ,
\quad \partial_t\psi \in L^2(0,T;{\mathcal  L}^2),
\quad |\psi|\leq 1~~\mbox{a.e.~in~\,}\Omega\times(0,T),\\
& {\bf A}\in C([0,T];{\bf L}^2)\cap 
L^{\infty}(0,T;{\bf H}_{\rm n}({\rm curl},{\rm div})) , \\
&
\partial_t{\bf A}\in L^2(0,T;{\bf L}^2),\quad
\nabla\times {\bf A}\in L^2(0,T;{\bf H}({\rm curl})),\quad
\nabla\cdot{\bf A}\in L^2(0,T;H^1),
\end{align*}
with $\psi(\cdot,0)=\psi_0 $, 
${\bf A}(\cdot,0)={\bf A}_0$, 
and the variational equations 
\begin{align}
&\int_0^T\bigg[\bigg(\eta\frac{\partial \psi}{\partial t} ,\varphi\bigg)
+ \bigg(\bigg(\frac{i}{\kappa} \nabla + \mathbf{A}\bigg)  \psi,
\bigg(\frac{i}{\kappa} \nabla + \mathbf{A}\bigg)\varphi\bigg)\bigg]\d t\nn\\
&\qquad\qquad\quad +\int_0^T\bigg[\bigg( (|\psi|^{2}-1) \psi
 -i\eta \kappa \psi \nabla\cdot{\bf A},\varphi\bigg)\bigg]\d t = 0,
\label{VPDE1}\\[15pt]
&\int_0^T\bigg[\bigg(\frac{\partial \mathbf{A}}{\partial t} ,{\bf a}\bigg)
+ \big(\nabla\times{\bf A},\nabla\times{\bf a}\big)
+\big(\nabla\cdot{\bf A},\nabla\cdot{\bf a}\big) \bigg]\d t   \nn\\
&=
 \int_0^T\bigg[\big({\bf H} , \nabla\times {\bf a}\big)
 -\bigg( {\rm Re}\bigg[\psi^*\bigg(\frac{i}{\kappa} \nabla 
+ \mathbf{A}\bigg) \psi\bigg],{\bf a}\bigg) \bigg]\d t,
\label{VPDE2}
\end{align}
hold for all $\varphi\in L^2(0,T;{\mathcal  H}^1)$ 
and ${\bf a}\in L^2(0,T;{\bf H}_{\rm n}({\rm curl},{\rm div}))$.
}
\end{definition}

The main result of this paper is the following theorem,
whose proof is presented in the next section.

\setcounter{theorem}{0}
\begin{theorem}\label{MainTHM1}
{\bf(Global well-posedness of the TDGL in a curved polyhedron)}$~$\\
{\it
For any given $T>0$, if 
$\psi_0\in {\mathcal  H}^1$,
$|\psi_0|\leq 1$ a.e. in $\Omega$,
${\bf A}_0\in {\bf H}_{\rm n}({\rm curl},{\rm div})$ and 
the external magnetic field ${\bf H}$ is a divergence-free
vector field such that
${\bf H} \in L^\infty(0,T;{\bf L}^2)\cap L^2(0,T;{\bf H}({\rm curl}))$, then 
the problem {\rm\refe{PDE1}-\refe{init}} 
has a unique weak solution
in the sense of Definition {\rm\ref{DefWSol}}.
}
\end{theorem}


\section{Proof of Theorem \ref{MainTHM1}}\label{WPness}
\setcounter{equation}{0}

In this section, we prove existence of 
a weak solution for the TDGL by constructing
approximating solutions and applying a compactness argument.
Uniqueness 
is then proved based on the regularity of the weak solution.
Unlike the two-dimensional case studied in \cite{LZ2}, 
the solution of the three-dimensional problem has weaker regularity,
which causes more difficulty in the compactness argument.
To overcome the difficulty, 
a different approach is used to construct approximating
solutions in order to control the order parameter pointwisely.
Uniqueness of the weak solution also 
needs to be proved under the weaker regularity.

To simplify the notations,
we denote by $C_{p_1,\cdots,p_l}$ a generic positive constant
which may be different at each occurrence,
depending on the parameters $p_1,\cdots,p_l$, 
and denote by $\epsilon$ a small generic positive constant.

\subsection{Some preliminary lemmas}
Here we present some preliminary
lemmas which will be used in the next subsection
in proving existence of weak solutions.
Lemma \ref{HnEbL3} is concerned with
the embedding of ${\bf H}_{\rm n}({\rm curl},{\rm div})$
into ${\bf L}^{3+\delta}$, which is 
used into control the magnetic potential, 
and Lemma \ref{UnBDPsi} is used to control the order
parameter pointwisely.

\begin{lemma}\label{HnEbL3}
There exists a positive constant $\delta$ 
{\rm(}which depends on the domain $\Omega${\rm)} such that
${\bf H}_{\rm n}({\rm curl},{\rm div})
\hookrightarrow {\bf L}^{3+\delta}$.
\end{lemma}
\noindent{\it Proof.}$\quad$ It is proved in \cite{BS87}
that any ${\bf u}\in {\bf H}_{\rm n}({\rm curl},{\rm div})$
admits a decomposition
${\bf u}={\bf v}+\nabla\varphi$, 
where $\varphi\in H^1, 
\Delta\varphi\in L^2, \,\,\mbox{and}\,\, \partial_n\varphi=0
\,\,\mbox{on}\,\, \partial\Omega$,
and 
$$\|{\bf v}\|_{{\bf H}^1}+\|\nabla\varphi\|_{{\bf L}^2}
+\|\Delta\varphi\|_{L^2}
\leq C\|{\bf u}\|_{{\bf H}_{\rm n}({\rm curl},{\rm div})} .
$$
From \cite{JK95} we know that 
$$
\|\nabla\varphi\|_{{\bf L}^{3+\delta}}
\leq C\|\Delta\varphi\|_{L^2}
$$
in any bounded Lipschitz domain $\Omega$, where $\delta$
depends on the domain.
It follows that 
$$
\|{\bf u}\|_{{\bf L}^{3+\delta}}
\leq \|{\bf v}\|_{{\bf L}^{3+\delta}}
+\|\nabla\varphi\|_{{\bf L}^{3+\delta}}
\leq C\|{\bf v}\|_{{\bf H}^1}
+C\|\Delta\varphi\|_{L^2}
\leq C\|{\bf u}\|_{{\bf H}_{\rm n}({\rm curl},{\rm div})} .
\qquad \qed\medskip
$$

\begin{lemma}\label{UnBDPsi}
{\it For any given ${\bf A}\in 
L^\infty(0,T;{\bf H}_{\rm n}({\rm curl},{\rm div}))$,  
the nonlinear equation {\rm\refe{PDE1}}
has a unique weak solution 
$\psi\in L^\infty(0,T;{\cal L}^\infty)\cap L^2(0,T;{\cal H}^1)
\cap H^1(0,T;({\mathcal H}^1)')$ 
in the sense of {\rm\refe{VPDE1}}.
Moreover, the solution  
satisfies that $|\psi|\leq 1$ a.e. in $\Omega\times(0,T)$.
}
\end{lemma}
\begin{proof}$~$
It suffices to construct approximating
solutions which preserve the pointwise estimate
\begin{align}\label{PWEst1}
\mbox{$|\psi|\leq 1$ a.e. in $\Omega\times(0,T)$.}
\end{align}
Since
${\bf H}_{\rm n}({\rm curl},{\rm div}) 
\hookrightarrow {\bf L}^{3+\delta} $,
it follows that ${\bf A}\in L^\infty(0,T;{\bf L}^{3+\delta})$.
Let ${\bf A}_m$, $m=1,2,\cdots$, 
be a sequence of smooth functions which 
converges to ${\bf A}$ in 
${\bf L}^{3+\delta}\cap {\bf H}_{\rm n}({\rm div}) $.
Since ${\bf A}_m$ is smooth for any fixed $m$.
By the theory of parabolic equations \cite{Evans},   
the equation 
\begin{align}\label{RegEqps}
&\eta\frac{\partial \psi_m}{\partial t} 
+ \bigg(\frac{i}{\kappa} \nabla + \mathbf{A}_m\bigg)^{2} \psi_m
 + (|\psi_m|^{2}-1) \psi_m
 -i\eta \kappa \psi_m \nabla\cdot{\bf A}_m = 0 
\end{align}
has a unique weak solution 
$\psi_m\in L^\infty(0,T;H^{1}\cap L^\infty)\cap H^1(0,T;L^2)$
in the sense of \refe{VPDE1}.
Let $(|\psi_m|^2-1)_+$ denote the positive part of $|\psi_m|^2-1$
and integrate this equation against $\psi_m^*(|\psi_m|^2-1)_+$. 
By considering the real part of the result, for any $t'\in(0,T)$ we derive
\begin{align*}
& \int_\Omega
\bigg(\frac{\eta}{4}\big(|\psi_m(x,t')|^2-1\big)_+ ^2\bigg)\d x
 + \int_0^{t'}\int_\Omega (|\psi_m|^{2}-1)^2_+ |\psi_m| ^2\d x\d t\\
&=-\int_0^{t'}{\rm Re}\int_\Omega 
\bigg(\frac{i}{\kappa} \nabla  \psi_m+ \mathbf{A}_m \psi\bigg)
\bigg(-\frac{i}{\kappa} \nabla 
+ \mathbf{A}_m\bigg)[\psi_m^* (|\psi_m|^2-1)_+]\d x\d t\\
&=-\int_0^{t'}\int_\Omega \bigg|\frac{i}{\kappa} 
\nabla  \psi_m+ \mathbf{A}_m \psi\bigg|^2
 (|\psi_m|^2-1)_+ \d x \d t\\
&\quad + \int_0^{t'}{\rm Re}\int_{\{|\psi_m|^2>1\}} 
\bigg(\frac{i}{\kappa} \nabla  \psi_m
+ \mathbf{A}_m \psi\bigg)\psi_m^*\bigg(\frac{i}{\kappa}  
\psi_m\nabla\psi_m^*+\frac{i}{\kappa}\psi_m^*\nabla\psi_m \bigg)\d x\d t\\
&=-\int_0^{t'}\int_\Omega \bigg|\frac{i}{\kappa} 
\nabla  \psi_m+ \mathbf{A}_m \psi_m\bigg|^2
 (|\psi_m|^2-1)_+ \d x\d t\\
&\quad -\int_0^{t'}{\rm Re}\int_{\{|\psi_m|^2>1\}}(|\psi_m|^2|\nabla\psi_m|^2
+ (\psi_m^*)^2\nabla\psi_m\cdot \nabla\psi_m )\d x\d t\\
& \leq 0,
\end{align*}
which implies that $\int_\Omega(|\psi_m(x,t')|^2-1)_+ ^2 \d x
=0$. Thus $|\psi_m|\leq 1$ a.e. in $\Omega\times(0,T)$.

Integrating \refe{RegEqps} against $\psi_m^*$
and considering the real part of the result, we derive that
\begin{align*}
\frac{\d}{\d t}\bigg(\frac{\eta}{2}\|\psi_m\|_{L^2}^2\bigg)
+ \big\|i\kappa^{-1} \nabla  \psi_m+ \mathbf{A}_m \psi_m\big\|_{L^2}^2
+\|\psi_m\|_{L^4}^4
&=\|\psi_m\|_{L^2}^2\leq C, 
\end{align*}
which further implies 
\begin{align}\label{PWEst2}
\|\psi_m\|_{L^2(0,T;{\mathcal H}^1)}\leq C,
\end{align}
where $C$ is independent of $m$. 
From \refe{RegEqps}
we also see that
\begin{align*}
\bigg|\int_0^T\big(\partial_t\psi_m,\varphi\big)\d t\bigg|
&=\bigg|\int_0^T\bigg[-\frac{1}{\kappa^2}\big(\nabla\psi_m,\nabla\varphi\big)
-\frac{i}{\kappa^2}\big(\nabla\psi_m,{\bf A}_m\varphi\big)
+\frac{i}{\kappa^2}\big({\bf A}_m\psi_m,\nabla\varphi\big) \\
&\qquad\quad\,\,\,\,
-\big(|{\bf A}_m|^2\psi_m,\varphi\big) 
-\big( (|\psi_m|^{2}-1) \psi_m,\varphi\big)
 +\big(i\eta \kappa \psi_m \nabla\cdot{\bf A}_m,\varphi\big)\bigg]\d t\bigg|\\
&\leq 
C\|\psi_m\|_{L^2(0,T;{\mathcal H}^1)}
\|\varphi\|_{L^2(0,T;{\mathcal H}^1)} \\
&\quad 
+C\|\psi_m\|_{L^2(0,T;{\mathcal H}^1)}
\|\mathbf{A}_m\|_{L^\infty(0,T;{\bf L}^3)}
\|\varphi\|_{L^2(0,T;{\mathcal L}^6)} \\
&\quad 
+C\|\psi_m\|_{L^2(0,T;{\mathcal L}^6)}
\|\mathbf{A}_m\|_{L^\infty(0,T;{\bf L}^3)}
\|\varphi\|_{L^2(0,T;{\mathcal H}^1)} \\
&\quad
+\|\mathbf{A}_m\|_{L^\infty(0,T;{\bf L}^3)}^2
\|\psi_m\|_{L^2(0,T;{\mathcal L}^6)}
\|\varphi\|_{L^2(0,T;{\mathcal L}^6)} \\
&\quad
+C\|(|\psi_m|^{2}-1) \psi_m\|_{L^\infty(0,T;{\cal L}^\infty)}
\|\varphi\|_{L^1(0,T;{\mathcal L}^1)} \\
&\quad
+C\|\nabla\cdot{\bf A}_m\|_{L^\infty(0,T;L^2)}
\|\varphi\|_{L^1(0,T;{\cal L}^2)} \\
&\leq C\|\varphi\|_{L^2(0,T;{\mathcal H}^1)} ,
\end{align*}
which implies (via a duality argument)
\begin{align}\label{PWEst3}
\|\partial_t\psi_m\|_{L^2(0,T;({\mathcal H}^1)')}
\leq C ,
\end{align}
where $C$ is independent of $m$.
With the estimates \refe{PWEst1}-\refe{PWEst3}, 
it is easy to prove that 
there exists 
$\psi\in L^\infty(0,T;{\mathcal L}^\infty)\cap 
L^2(0,T;{\mathcal H}^1)\cap H^1(0,T;({\mathcal H}^1)')$ 
and a subsequence of $\psi_m$, $m=1,2,\cdots$,
which converges to $\psi$ weakly in
$L^2(0,T;{\mathcal H}^1)\cap H^1(0,T;({\mathcal H}^1)')$,
strongly in $L^p(0,T;{\mathcal L}^p)$ for any $1<p<\infty$,
and convergence is pointwise a.e. in $\Omega\times(0,T)$.
It is easy to see that $\psi$ is
a weak solution of \refe{PDE1}, 
with $|\psi|\leq 1$ a.e. in $\Omega\times(0,T)$.

Uniqueness of the weak solution can be proved easily
based on the regularity of
$\psi$. 
\end{proof}

\subsection{Construction of approximating solutions}
In this subsection, we construct approximating solutions
of \refe{PDE1}-\refe{init} in semi-finite dimensional spaces.

Let $M:{\bf H}_{\rm n}({\rm curl},{\rm div})
\rightarrow ({\bf H}_{\rm n}({\rm curl},{\rm div}))'$ be defined by
$$
(M{\bf u},{\bf v})=({\bf u},{\bf v})+(\nabla\times {\bf u},\nabla\times {\bf v})
+(\nabla\cdot{\bf u},\nabla\cdot{\bf v}) ,
\quad\mbox{for}~ {\bf u},{\bf v}
\in {\bf H}_{\rm n}({\rm curl},{\rm div}) \, .
$$
It is easy to see that the bilinear form on the right-hand side 
is coercive on ${\bf H}_{\rm n}({\rm curl},{\rm div})$, 
which is compactly embedded into ${\bf L}^2$
(see \cite{Webner}).
In this case,
the spectrum of the operator $M$ consists of 
eigenvalues $\lambda_1,\lambda_2,\cdots,\lambda_N,\cdots$, 
which tend to infinity, and the 
corresponding eigenvectors ${\bf a}_1,{\bf a}_2,{\bf a}_3,\cdots$ 
form a Hilbert basis of ${\bf H}_{\rm n}({\rm curl},{\rm div})$
(see Theorem 2.37 of \cite{McLean00}).
Let ${\bf X}_N={\rm span}\{{\bf a}_1,{\bf a}_2,\cdots,{\bf a}_N\}$,
which is a finite dimensional subspace of 
${\bf H}_{\rm n}({\rm curl},{\rm div})$,
equipped with the norm of ${\bf H}_{\rm n}({\rm curl},{\rm div})$.
We look for $\Psi_N(t)\in {\mathcal H}^1$ and
${\bf \Lambda}_N(t)\in {\bf X}_N$
such that
\begin{align}
&\bigg(\eta\frac{\partial \Psi_N}{\partial t} ,\varphi\bigg)
+ \bigg(\bigg(\frac{i}{\kappa} \nabla +{\bf \Lambda}_N\bigg)  \Psi_N,
\bigg(\frac{i}{\kappa} \nabla + {\bf \Lambda}_N\bigg)\varphi\bigg)\nn\\
&\qquad\quad +\bigg( (|\Psi_N|^{2}-1) \Psi_N-i\eta\kappa
 \Psi_N \nabla\cdot{\bf\Lambda}_N,\varphi\bigg) = 0,
\label{DVPDE1}\\[8pt]
&\bigg(\frac{\partial {\bf \Lambda}_N}{\partial t} ,{\bf a}\bigg)
+ \big(\nabla\times{\bf \Lambda}_N,\nabla\times{\bf a}\big)
+\big(\nabla\cdot{\bf \Lambda}_N,\nabla\cdot{\bf a}\big)  \nn\\
&\qquad\quad 
+ \bigg( {\rm Re}\bigg[\Psi_N^*\bigg(\frac{i}{\kappa} \nabla 
+{\bf \Lambda}_N\bigg) \Psi_N\bigg],{\bf a}\bigg) =
 \big({\bf H} , \nabla\times {\bf a}\big),
\label{DVPDE2}
\end{align}
for any $\varphi\in {\mathcal H}^1$ and ${\bf a}\in {\bf X}_N$ 
at any $t\in (0,T)$,
with the initial conditions $\Psi(0) = \psi_{0}\in {\cal H}^1$ and  
${\bf\Lambda}(0) = \Pi_N{\bf A}_{0}\in{\bf H}_n({\rm curl,div})$, where 
$\Pi_N$ is the 
projection of ${\bf H}_{\rm n}({\rm curl},{\rm div})$ onto its subspace 
${\bf X}_N$. We have the following lemma
concerning the approximating solution $(\Psi_N,{\bf\Lambda}_N)$.

\begin{lemma}\label{ExAppSol}
For a given positive integer $N$,
the system \refe{DVPDE1}-\refe{DVPDE2} has a unique solution 
\begin{align*}
&\Psi_N\in L^\infty(0,T;{\mathcal H}^1)\cap H^1(0,T;{\mathcal L}^2)
\cap L^\infty(\Omega\times(0,T)) ,
\quad |\Psi_N|\leq 1\,\,\, a.e.\,\,\, in \,\,\,\Omega\times(0,T),\\
&{\bf\Lambda}_N\in W^{1,\infty}(0,T;{\bf X}_N)
\end{align*}
Moreover, we have the uniform estimate:
\begin{align}\label{unifEstAS}
&\|\Psi_N\|_{L^\infty(0,T;{\mathcal  H}^1)}
+\|\partial_t\Psi_N\|_{L^2(0,T;{\mathcal  L}^2)}
+\|{\bf\Lambda}_N\|_{L^\infty(0,T;{\bf H}_{\rm n}({\rm curl},{\rm div}))}
+\|\partial_t{\bf\Lambda}_N\|_{L^2(0,T;L^2)}
\leq C,
\end{align}
where the constant $C$ does not depend on $N$.
\end{lemma}

To prove Lemma \ref{SchLem}, we apply 
Schaefer's fixed point theorem \cite{Evans}:

\begin{lemma}\label{SchLem}
$\,\,${\bf(Schaefer's fixed point theorem)}$\,\,\,$
Let ${\cal N} : L^\infty(0,T;{\bf X}_N) \rightarrow L^\infty(0,T;{\bf X}_N)$ be a 
continuous and compact mapping such that the set
\begin{align}\label{SetVN}
{\bf V}_N:=\{{\bf \Lambda}_N^0\in L^\infty(0,T;{\bf X}_N):
\exists \, s\in[0,1]\mbox{ satisfying }
{\bf \Lambda}_N^0= s {\cal N}{\bf \Lambda}_N^0\}
\end{align}
is bounded in $L^\infty(0,T;{\bf X}_N)$. 
Then the mapping ${\cal N}$ has
a fixed point in $L^\infty(0,T;{\bf X}_N)$: 
there exists ${\bf \Lambda}\in L^\infty(0,T;{\bf X}_N)$ 
satisfying ${\bf \Lambda}=  {\cal N}{\bf \Lambda}$.
\end{lemma}

\noindent{\it Proof of Lemma \ref{ExAppSol}.}$\quad$ 
For a given ${\bf \Lambda}_N^0\in L^\infty(0,T;{\bf X}_N)$, we 
define $\Psi_N\in L^\infty(0,T;L^\infty)\cap L^2(0,T;{\cal H}^1)$
as the unique weak solution of the equation
\begin{align}
&\bigg(\eta\frac{\partial \Psi_N}{\partial t} ,\varphi\bigg)
+ \bigg(\bigg(\frac{i}{\kappa} \nabla +{\bf \Lambda}_N^0\bigg)  \Psi_N,
\bigg(\frac{i}{\kappa} \nabla + {\bf \Lambda}_N^0\bigg)\varphi\bigg)\nn\\
&\qquad\quad +\bigg( (|\Psi_N|^{2}-1) \Psi_N-i\eta\kappa
 \Psi_N \nabla\cdot{\bf\Lambda}_N^0,\varphi\bigg) = 0 ,
 \qquad\forall\,\varphi\in L^2(0,T;{\cal H}^1).
\label{DVPDE01}
\end{align}
Existence and uniqueness of the solution $\Psi_N$
follow Lemma \ref{UnBDPsi}.
Then we define ${\bf \Lambda}_N $ as the solution of 
\begin{align}
&\bigg(\frac{\partial {\bf \Lambda}_N}{\partial t} ,{\bf a}\bigg)
+ \big(\nabla\times{\bf \Lambda}_N,\nabla\times{\bf a}\big)
+\big(\nabla\cdot{\bf \Lambda}_N,\nabla\cdot{\bf a}\big)  \nn\\
&\qquad\quad 
+ \bigg( {\rm Re}\bigg[\Psi_N^*\bigg(\frac{i}{\kappa} \nabla 
+{\bf \Lambda}_N^0\bigg) \Psi_N\bigg],{\bf a}\bigg) =
 \big({\bf H}, \nabla\times {\bf a}\big),
\label{DVPDE02}
\end{align}
Since ${\bf X}_N$ is a finite dimensional space,
the equation above can be reduced to an ordinary differential equation.
Therefore, existence and uniqueness of a solution 
${\bf \Lambda}_N\in W^{1,\infty}(0,T;{\bf X}_N)$ 
are obvious.
We denote the mapping from 
${\bf \Lambda}_N^0\in L^\infty(0,T;{\bf X}_N)$
to $\Psi_N\in L^\infty(0,T;{\cal L}^\infty)\cap L^2(0,T;{\cal H}^1)$
by ${\cal N}_1$, and denote the mapping from 
$\Psi_N\in L^\infty(0,T;{\cal L}^\infty)\cap L^2(0,T;{\cal H}^1)$
to ${\bf \Lambda}_N\in L^\infty(0,T;{\bf X}_N)$ by ${\cal N}_2$,
and then define ${\cal N}:={\cal N}_1{\cal N}_2$.
We shall prove that the mapping ${\cal N}$
satisfies the conditions of Lemma \ref{SchLem} 
and thus has a fixed point in $L^\infty(0,T;{\bf X}_N)$.

Firstly, we prove that 
the mapping ${\cal N}:L^\infty(0,T;{\bf X}_N)\rightarrow 
L^\infty(0,T;{\bf X}_N)$ is 
continuous and compact.
To prove the continuity of the mapping, we
let $\Psi_N={\cal N}_1{\bf \Lambda}_N^0$,  
$\widetilde \Psi_N={\cal N}_1\widetilde{\bf \Lambda}_N^0$,
${\bf\Lambda}_N={\cal N}{\bf \Lambda}_N^0$
and $\widetilde{\bf\Lambda}_N={\cal N}\widetilde{\bf \Lambda}_N^0$,
and assume that
${\bf \Lambda}_N^0$ and $\widetilde{\bf \Lambda}_N^0$
are bounded in $L^\infty(0,T;{\bf X}_N)$:
\begin{align}\label{BDLL}
\|{\bf \Lambda}_N^0\|_{L^\infty(0,T;{\bf X}_N)}+
\|\widetilde{\bf \Lambda}_N^0\|_{L^\infty(0,T;{\bf X}_N)}\leq K,
\end{align}
where $K$ is some positive constant.
From Lemma \ref{UnBDPsi} we see that 
the inequality above implies 
\begin{align}
&\|\Psi_N\|_{L^2(0,T;{\cal H}^1)}
+\|\widetilde\Psi_N\|_{L^2(0,T;{\cal H}^1)}\leq C_K, \label{BDPSi1}\\
&|\Psi_N|\leq 1\quad\mbox{and}\quad |\widetilde\Psi_N|\leq 1,\quad
\mbox{a.e. in }\Omega\times(0,T) . \label{BDPSi2}
\end{align}
If we define $e_N=\widetilde\Psi_N-\Psi_N$ and
${\bf E}^0_N=\widetilde{\bf \Lambda}_N^0-{\bf \Lambda}_N^0$, 
then \refe{DVPDE01} implies
\begin{align}
&\int_0^T\Big[\big(\eta\partial_t e_N  ,\varphi\big)
+ \frac{1}{\kappa^2}\big(\nabla  e_N, \nabla\varphi\big) 
+ \big(|\widetilde{\bf \Lambda}_N^0|^2   e_N,  \varphi\big) \Big]\d t\nn\\
& =\int_0^T\Big[-\frac{i}{\kappa}\big(\widetilde{\bf \Lambda}_N^0
\cdot\nabla e_N ,\varphi\big)
-\frac{i}{\kappa}\big({\bf E}_N^0\cdot\nabla  \Psi_N ,\varphi\big)
+\frac{i}{\kappa}\big(e_N \widetilde{\bf \Lambda}_N^0,\nabla\varphi\big)
+\frac{i}{\kappa}\big( \Psi_N {\bf E}^0_N,\nabla\varphi\big)  \nn\\
&\quad - \big((|\widetilde{\bf \Lambda}_N^0|^2
 -|{\bf \Lambda}_N^0|^2) \Psi_N,  \varphi\big)
 -\big( (|\widetilde \Psi_N|^{2}-1)\widetilde \Psi_N
 -(|\Psi_N|^{2}-1) \Psi_N,\varphi\big)\Big]\d t\nn\\
&\quad -\int_0^T\big(i\eta\kappa\widetilde \Psi_N\nabla\cdot{\bf E}^0_N
+i\eta\kappa e_N\nabla\cdot{\bf \Lambda}_N^0,\varphi\big)\d t .
\nn
\end{align}
Substituting $\varphi(x,t)=e_N(x,t) 1_{(0,t')}(t)$
into the equation above and considering
the real part, we obtain
\begin{align*}
& \frac{\eta}{2} \|e_N(\cdot,t') \|_{{\cal L}^2}^2 
+ \int_0^{t'}\Big(\frac{1}{\kappa^2}\|\nabla  e_N\|_{{\cal L}^2}^2
+  \|\widetilde{\bf \Lambda}_N^0    e_N\|_{{\cal L}^2}^2\Big)\d t \\
&\leq 
\int_0^{t'}\Big(C\|\widetilde{\bf \Lambda}_N^0\|_{{\bf L}^{3+\delta}}
\|\nabla e_N\|_{{\cal L}^2}
 \|e_N\|_{{\cal L}^{6-4\delta/(1+\delta)}}
+C\|{\bf E}_N^0\|_{{\bf L}^{3+\delta}}\|\nabla \Psi_N\|_{{\cal L}^2}
\|e_N\|_{{\cal L}^{6-4\delta/(1+\delta)}}\\
&\quad
+C\| e_N\|_{{\cal L}^{6-4\delta/(1+\delta)}}
 \|\widetilde{\bf \Lambda}_N^0\|_{{\bf L}^{3+\delta}}
\|\nabla e_N\|_{{\cal L}^2}  \\
&\quad +C\| \Psi_N\|_{{\cal L}^\infty}\| {\bf E}^0_N\|_{{\bf L}^2}
\|\nabla e_N\|_{{\cal L}^2}
+C(\|\widetilde{\bf \Lambda}_N^0\|_{{\bf L}^{3+\delta}}
+\|{\bf \Lambda}_N^0\|_{{\bf L}^{3+\delta}})
\|{\bf E}_N^0\|_{{\bf L}^2} \| \Psi_N\|_{{\cal L}^\infty}
\|e_N\|_{{\cal L}^{6-4\delta/(1+\delta)}} \\
&\quad 
+C\| e_N\|_{{\cal L}^2}^2
+C\| \widetilde\Psi_N\|_{{\cal L}^\infty}\|\nabla\cdot{\bf E}_N^0\|_{L^2}
\|e_N\|_{{\cal L}^2}\Big)\d t\\
&\leq \int_0^{t'}\Big(CK\|\nabla e_N\|_{{\cal L}^2}
 (C_\epsilon \|e_N\|_{{\cal L}^2}+\epsilon\|\nabla e_N\|_{{\cal L}^2})\\
&\quad 
 +C\|\nabla \Psi_N\|_{{\cal L}^2}\|{\bf E}_N^0\|_{{\bf H}_{\rm n}({\rm curl},{\rm div})}
 (C_\epsilon \|e_N\|_{{\cal L}^2}+\epsilon\|\nabla e_N\|_{{\cal L}^2})\\
&\quad
+CK\|\nabla e_N\|_{{\cal L}^2}(C_\epsilon \|e_N\|_{{\cal L}^2}
+\epsilon\|\nabla e_N\|_{{\cal L}^2}) 
+C\| {\bf E}_N^0\|_{L^2}\|\nabla e_N\|_{{\cal L}^2}\\
&\quad
+CK\|{\bf E}_N^0\|_{L^2}(C_\epsilon \|e_N\|_{{\cal L}^2}+\epsilon\|\nabla e_N\|_{L^2})  +C\| e_N\|_{L^2}^2
+C\|\nabla\cdot {\bf E}_N^0\|_{{\bf L}^2}\| e_N\|_{{\cal L}^2} \Big)\d t\\  
&\leq \int_0^{t'}\Big((1+CK)\epsilon\|\nabla e_N\|_{{\cal L}^2}^2
+C_\epsilon\|e_N\|_{{\cal L}^2}^2
+
(C_{\epsilon,K}+C_{\epsilon,K}\|\nabla \Psi_N\|_{{\cal L}^2}^2)
\|{\bf E}_N^0\|_{{\bf H}_{\rm n}({\rm curl},{\rm div})}^2\Big)\d t ,
\end{align*}
which holds for any $\epsilon\in(0,1)$.
Since $e_N(x,0)=0$,
by choosing $\epsilon<1/(2\kappa^2+2CK\kappa^2)$ in
the inequality above, 
applying Gronwall's inequality and
using \refe{BDPSi1}, we derive 
\begin{align}\label{ErreN}
& \|e_N\|_{L^\infty(0,T;{\cal L}^2)}^2
+\|e_N\|_{L^2(0,T;{\cal H}^1)}^2 \nn\\
&\leq e^{C_KT}
\bigg(\int_0^T(C_{K}+C_{K}\|\nabla \Psi_N\|_{{\cal L}^2}^2)
\d t \bigg)\|{\bf E}_N^0\|_{L^\infty(0,T;{\bf H}_n({\rm curl,div}))}^2 \nn\\
&\leq 
C_K\|{\bf E}_N^0\|_{L^\infty(0,T;{\bf H}_n({\rm curl,div}))}^2 .
\end{align}

If we define
${\bf E}_N=\widetilde{\bf \Lambda}_N-{\bf \Lambda}_N$, 
then \refe{DVPDE02} implies
\begin{align}
&\int_0^T\Big[\big(\partial_t{\bf E}_N ,{\bf a}\big)
+ \big(\nabla\times{\bf E}_N,\nabla\times{\bf a}\big)
+\big(\nabla\cdot{\bf E}_N,\nabla\cdot{\bf a}\big) \Big]\d t\nn\\
& =-\int_0^T{\rm Re} \bigg( 
\frac{i}{\kappa}(\widetilde\Psi_N^*\nabla\widetilde\Psi_N
-\Psi_N^*\nabla\Psi_N)
+  \widetilde{\bf \Lambda}^0_N(|\widetilde\Psi_N|^2-|\Psi_N|^2)
+|\Psi_N|^2 {\bf E}_N\, ,\, {\bf a}\bigg) \d t .
\end{align}
Since ${\bf X}_N$ is a finite dimensional space,
any two norms on ${\bf X}_N$ 
are equivalent. This implies that
$\|{\bf E}_N\|_{{\bf H}_{\rm n}({\rm curl},{\rm div})}
\leq C_N\|{\bf E}_N\|_{{\bf L}^2}$.
Substituting ${\bf a}(x,t)={\bf E}_N(x,t)1_{(0,t')}(t)$
into the equation above, we obtain
\begin{align*}
& \frac{1}{2}\|{\bf E}_N(\cdot,t')\|_{L^2}^2 
+\int_0^{t'}\Big(\|\nabla\times{\bf E}_N \|_{{\bf L}^2}^2
+\|\nabla\cdot{\bf E}_N \|_{{\bf L}^2}^2 \Big)\d t\\
&\leq \int_0^{t'}\Big(C \| e_N\|_{{\cal L}^{6-4\delta/(1+\delta)}}
\|\nabla \widetilde\Psi_N\|_{{\cal L}^2}\| {\bf E}_N\|_{{\bf L}^{3+\delta}}
+C\|\nabla  e_N \|_{{\cal L}^2} \| {\bf E}_N\|_{{\bf L}^2} \\
&\qquad\qquad
+ C\|\widetilde{\bf \Lambda}^0_N\|_{{\bf L}^{3+\delta}}
\|e_N\|_{{\cal L}^{6-4\delta/(1+\delta)}}\|{\bf E}_N\|_{{\bf L}^2}
+\|{\bf E}_N\|_{{\bf L}^2}^2\Big)\d t\\
&\leq C \| e_N\|_{L^2(0,T;{\cal H}^1)} 
\|\widetilde\Psi_N\|_{L^2(0,T;{\cal H}^1)} 
\| {\bf E}_N\|_{L^\infty(0,T;{\bf H}_{\rm n}({\rm curl},{\rm div}))}
+\| e_N \|_{L^2(0,T;{\cal H}^1)} \| {\bf E}_N\|_{L^2(0,T;{\bf L}^2)} \\
&\quad 
+\|\widetilde{\bf \Lambda}^0_N\|_{L^\infty(0,T;{\bf L}^{3+\delta})}
 \| e_N\|_{L^2(0,T;{\cal H}^1)} \|{\bf E}_N\|_{L^2(0,T;{\cal L}^2)}
 +\|{\bf E}_N\|_{L^2(0,T;{\cal L}^2)}^2 \\
&\leq
C_NC_K\|  {\bf E}^0_N\|_{L^\infty(0,T;{\bf H}_{\rm n}({\rm curl},{\rm div}))} 
\| {\bf E}_N\|_{L^\infty(0,T;{\bf L}^2)}
+C\|  {\bf E}^0_N\|_{L^\infty(0,T;{\bf H}_{\rm n}({\rm curl},{\rm div}))}
\| {\bf E}_N\|_{L^2(0,T;{\bf L}^2)} \\
&\quad 
+CK\|  {\bf E}^0_N\|_{L^\infty(0,T;{\bf H}_{\rm n}({\rm curl},{\rm div}))}
\| {\bf E}_N\|_{L^2(0,T;{\bf L}^2)}
+\|{\bf E}_N\|_{L^2(0,T;{\cal L}^2)}^2 \\
&\leq C_NC_KC_\epsilon
\|  {\bf E}^0_N\|_{L^\infty(0,T;{\bf H}_{\rm n}({\rm curl},{\rm div}))}^2
+\epsilon \|  {\bf E} _N\|_{L^\infty(0,T;{\bf L}^2)} ^2
+ C\|  {\bf E} _N\|_{L^2(0,T;{\bf L}^2)} ^2
\end{align*}
Since ${\bf E}_N(x,0)=0$,
by applying Gronwall's inequality we derive 
\begin{align*} 
\|{\bf E}_N\|_{L^\infty(0,T;{\bf L}^2)}
\leq 
C_NC_K\|{\bf E}_N^0\|_{L^\infty(0,T;{\bf H}_n({\rm curl,div}))} .
\end{align*}
which further implies that
\begin{align}\label{ErrEN}
\|{\bf E}_N\|_{L^\infty(0,T;{\bf H}_n({\rm curl,div}))}
\leq 
C_NC_K\|{\bf E}_N^0\|_{L^\infty(0,T;{\bf H}_n({\rm curl,div}))} .
\end{align}
This prove that 
the mapping ${\cal N}:L^\infty(0,T;{\bf X}_N)\rightarrow 
L^\infty(0,T;{\bf X}_N)$ is 
continuous. Since ${\bf X}_N$ is finite dimensional,
any continuous mapping is also compact.

Secondly, we prove that the set ${\bf V}_N$
defined in \refe{SetVN}
is bounded in $L^\infty(0,T;{\bf X}_N)$. Suppose that 
${\bf \Lambda}_N^0\in {\bf V}_N$, then we define
$\Psi_N={\cal N}_1{\bf \Lambda}_N^0$
and ${\bf \Lambda}_N={\cal N}{\bf \Lambda}_N^0$,
which
satisfy ${\bf \Lambda}_N^0=s{\bf \Lambda}_N $ 
and the following equations:
\begin{align}
&\bigg(\eta\frac{\partial \Psi_N}{\partial t} ,\varphi\bigg)
+ \bigg(\bigg(\frac{i}{\kappa} \nabla + s{\bf \Lambda}_N\bigg)  \Psi_N,
\bigg(\frac{i}{\kappa} \nabla + s {\bf \Lambda}_N\bigg)\varphi\bigg)\nn\\
&\qquad\quad +\bigg( (|\Psi_N|^{2}-1) \Psi_N-is \eta\kappa
 \Psi_N \nabla\cdot{\bf\Lambda}_N,\varphi\bigg) = 0,
\label{DVPDE21}\\[8pt]
&\bigg(\frac{\partial {\bf \Lambda}_N}{\partial t} ,{\bf a}\bigg)
+ \big(\nabla\times{\bf \Lambda}_N,\nabla\times{\bf a}\big)
+\big(\nabla\cdot{\bf \Lambda}_N,\nabla\cdot{\bf a}\big)  \nn\\
&\qquad\quad 
+ \bigg( {\rm Re}\bigg[\Psi_N^*\bigg(\frac{i}{\kappa} \nabla 
+ s{\bf \Lambda}_N\bigg) \Psi_N\bigg],{\bf a}\bigg) =
 \big({\bf H} , \nabla\times {\bf a}\big) .
\label{DVPDE22}
\end{align}
It suffices to prove the
boundedness of ${\bf \Lambda}_N$ in $L^\infty(0,T;{\bf X}_N)$. 
Substituting $\varphi=\Psi_N$ into \refe{DVPDE21}
and considering the real part,
we get
\begin{align*}
&\frac{\d}{\d t}\bigg(\frac{\eta}{2}\|\Psi_N\|_{{\cal L}^2}^2\bigg)
+\bigg\|\bigg(\frac{i}{\kappa} \nabla + s{\bf \Lambda}_N\bigg)  \Psi_N
\bigg\|_{{\cal L}^2}^2 \leq \|\Psi_N\|_{{\cal L}^2}^2  \leq C, 
\end{align*}
which implies 
\begin{align}
&\|\Psi_N\|_{L^\infty(0,T;{\cal L}^2)}^2 
+\bigg\|\bigg(\frac{i}{\kappa} \nabla + s{\bf \Lambda}_N\bigg)  \Psi_N
\bigg\|_{L^2(0,T;{\cal L}^2)}^2 \leq C .
\end{align}
Substituting 
${\bf a}={\bf\Lambda}_N$ into \refe{DVPDE22},
we obtain
\begin{align}
&\frac{\d}{\d t}\bigg(\frac{1}{2}\|{\bf \Lambda}_N\|_{{\bf L}^2}^2\bigg)
+ \|\nabla\times{\bf \Lambda}_N\|_{{\bf L}^2}^2
+\|\nabla\cdot{\bf \Lambda}_N\|_{{\bf L}^2}^2  \nn\\
& =
 \big({\bf H} , \nabla\times {\bf \Lambda}_N\big)
 -\bigg( {\rm Re}\bigg[\Psi_N^*\bigg(\frac{i}{\kappa} \nabla 
+ s{\bf \Lambda}_N\bigg) \Psi_N\bigg],{\bf \Lambda}_N\bigg) \nn\\
&\leq \|{\bf H}\|_{{\bf L}^2}\|\nabla\times{\bf \Lambda}_N\|_{{\bf L}^2}
+ \bigg\|\bigg(\frac{i}{\kappa} \nabla + s{\bf \Lambda}_N\bigg)  \Psi_N
\bigg\|_{{\cal L}^2}\|{\bf \Lambda}_N\|_{{\bf L}^2} \nn\\
&\leq  \frac{1}{2}\|{\bf H}\|_{{\bf L}^2}^2
+ \frac{1}{2}\|\nabla\times{\bf \Lambda}_N\|_{{\bf L}^2}^2
+  \frac{1}{2}\bigg\|\bigg(\frac{i}{\kappa} \nabla
 + s{\bf \Lambda}_N\bigg)  \Psi_N\bigg\|_{{\cal L}^2}^2
+ \frac{1}{2}\|{\bf \Lambda}_N\|_{{\bf L}^2}^2 \nn\\
&\leq  C+\frac{1}{2}\|{\bf H}\|_{{\bf L}^2}^2
+ \frac{1}{2}\|\nabla\times{\bf \Lambda}_N\|_{{\bf L}^2}^2
+ \frac{1}{2}\|{\bf \Lambda}_N\|_{{\bf L}^2}^2, \nn
\end{align}
which implies (via Gronwall's inequality)
\begin{align}
&\|{\bf\Lambda}_N\|_{L^\infty(0,T;{\cal L}^2)}    \leq C .
\end{align}
Since ${\bf X}_N$ is finite dimensional, we have
\begin{align}
\|{\bf\Lambda}_N\|_{L^\infty(0,T;{\bf X}_N)}  
\leq C_N\|{\bf\Lambda}_N\|_{L^\infty(0,T;{\cal L}^2)}    \leq C_N .
\end{align}
Therefore, the set ${\bf V}_N$ defined in \refe{SetVN}
is bounded in $L^\infty(0,T;{\bf X}_N)$.

To conclude, we have proved that the mapping
${\cal N}$ satisfies the conditions of Lemma \ref{SchLem},
which implies that
the mapping ${\cal N}$ has a fixed point 
${\bf\Lambda}_N\in L^\infty(0,T;{\bf X}_N)$.
If we define
$\Psi_N={\cal N}_1{\bf \Lambda}_N$, then 
$(\Psi_N, {\bf\Lambda}_N)$ is a solution of 
\refe{DVPDE1}-\refe{DVPDE2}. 
It remains to present estimates on the 
regularity of $\Psi_N$ and $ {\bf\Lambda}_N$
(uniformly with respect to $N$).

Substituting $\varphi=\partial_t\Psi$ and 
${\bf a}=\partial_t{\bf\Lambda}$ into the equations, 
we obtain 
\begin{align*}
&\frac{\d}{\d t}\int_\Omega\frac{1}{2}\bigg(
\bigg|\frac{i}{\kappa}\nabla\Psi_N+{\bf \Lambda}_N\Psi_N\bigg|^2
+\frac{1}{2}(|\Psi_N|^2-1)^2
+|\nabla\times{\bf \Lambda}_N-{\bf H}|^2
+|\nabla\cdot{\bf \Lambda}_N|^2 \bigg)\d x\\
&~~~ +\int_\Omega\bigg(\bigg|
\frac{\partial {\bf \Lambda}_N}{\partial t}\bigg|^2
+\eta\bigg|\frac{\partial \Psi_N}{\partial t}\bigg|^2\bigg)\d x \\
&=\eta\kappa\int_\Omega {\rm Im}\bigg( \Psi_N 
\frac{\partial \Psi_N^*}{\partial t}\bigg)
\nabla\cdot{\bf \Lambda}_N \, \d x \\
&\leq \frac{1}{2}\int_\Omega 
\eta\bigg|\frac{\partial \Psi_N}{\partial t}\bigg|^2 \d x 
+\frac{1}{2}\int_\Omega 
\eta\kappa^2|\nabla\cdot{\bf \Lambda}_N|^2 \d x ,
\end{align*}
which implies \refe{unifEstAS}
via Gronwall's inequality.
The proof of Lemma \ref{ExAppSol} is completed.
\qed

\subsection{Existence of a weak solution}

We show that a subsequence of the approximating
solutions constructed in the last subsection
converges to a weak solution of \refe{PDE1}-\refe{PDE2}.
We need the following lemma \cite{Lions69}.
\begin{lemma}\label{ALLem}
$\!\!${\bf(Aubin--Lions)}$\;$ {\it Let
$B_1\hookrightarrow\hookrightarrow B_2\hookrightarrow B_3$ be
reflexive and separable Banach spaces. Then
$$\{u\in L^p(I;B_1)|\;u_t\in
L^q(I;B_3)\}\hookrightarrow\hookrightarrow L^p(I;B_2), \quad
1<p,q<\infty,$$ 
where the symbol ``$\hookrightarrow\hookrightarrow$''
indicates compact embedding.}
\end{lemma}\medskip

On one hand, by choosing $B_1={\cal H}^1$, $B_2=B_3={\cal L}^2$ 
and $1<p=q<\infty$ in Lemma \ref{ALLem},
from \refe{unifEstAS} 
we see that the set $\{\Psi_N:N=1,2,\cdots\}$
is compact in $L^p(0,T;{\cal L}^2)$.
Therefore, there exists a subsequence $\Psi_{N_m}$,
$m=1,2,\cdots$, which converges to a function
$\psi$ in $L^p(0,T;{\cal L}^2)$ and the convergence is also pointwise
a.e. in $\Omega\times(0,T)$.
Since $|\Psi_{N_m}|\leq 1$ a.e. in $\Omega\times(0,T)$, it follows that
$|\psi|\leq 1$ a.e. in $\Omega\times(0,T)$.
By the Lebesgue dominated convergence theorem,
the sequence $\Psi_{N_m}$ also converges to  
$\psi$ in $L^p(0,T;{\cal L}^p)$.
On the other hand, by the Eberlein-Shmulyan theorem 
(see Page 141 of \cite{Yosida}),
by passing to a subsequence if necessary, 
$\Psi_{N_m}$ converges to $\psi$ weakly in 
$L^p(0,T;{\cal H}^1)$ 
and weakly$^*$ in $L^\infty(0,T;{\cal H}^1)$, and
$\partial_t\Psi_{N_m}$ converges to $\partial_t\psi$ weakly in 
$L^2(0,T;{\cal L}^2)$.
In a similar way, by passing to a subsequence if necessary, 
we can derive
the convergence of ${\bf\Lambda}_{N_m}$
to a function ${\bf A}\in L^\infty(0,T;{\bf H}_{\rm n}({\rm curl},{\rm div}))
\cap H^1(0,T;{\bf L}^2)$.
For the reader's convenience, we summarize the
convergence of $\Psi_{N_m}$ and ${\bf\Lambda}_{N_m}$ below:
there exist
\begin{align*}
& \psi\in L^\infty(0,T;{\mathcal  H}^1)
\cap H^1(0,T;{\mathcal  L}^2),\quad |\psi|\leq 1\,\,\, 
\mbox{a.e. in $\Omega\times(0,T)$},\\
&{\bf A}\in L^\infty(0,T;{\bf H}_{\rm n}({\rm curl},{\rm div}))
\cap H^1(0,T;{\bf L}^2),
\end{align*}
such that 
\begin{align*}
& \Psi_{N_m} \rightharpoonup
\psi\quad\mbox{weakly$^*$ in}~~ L^\infty(0,T;{\mathcal  H}^1) ,\\
& \Psi_{N_m} \rightharpoonup
\psi\quad\mbox{weakly in}~~ L^p(0,T;{\mathcal  H}^1) ~~
\mbox{for any $1<p<\infty$} ,\\
&\partial_t\Psi_{N_m} \rightharpoonup
\partial_t\psi\quad\mbox{weakly in}~~ L^2(0,T;{\mathcal  L}^2) ,\\
& \Psi_{N_m} \rightarrow
\psi\quad \mbox{strongly in}~~ L^{p}(0,T;{\mathcal  L}^{p})~~
\mbox{for any $1<p<\infty$ } ,\\
& {\bf\Lambda}_{N_m} \rightharpoonup
{\bf A}\quad\mbox{weakly$^*$ in}~~ 
L^\infty(0,T;{\bf H}_{\rm n}({\rm curl},{\rm div})) ,\\
& {\bf\Lambda}_{N_m} \rightharpoonup
{\bf A}\quad\mbox{weakly in}~~ 
L^p(0,T;{\bf H}_{\rm n}({\rm curl},{\rm div})) ~~
\mbox{for any $1<p<\infty$} ,\\
&\partial_t{\bf\Lambda}_{N_m} \rightharpoonup
\partial_t{\bf A}\quad\mbox{weakly in}~~ L^2(0,T;{\bf L}^2) ,\\
& {\bf\Lambda}_{N_m} \rightarrow
{\bf A}\quad \mbox{strongly in}~~ L^p(0,T;{\bf L}^{3+\delta})
~~\mbox{for any $1<p<\infty$} .
\end{align*}
It is easy to see that the convergences described above imply
\begin{align*}
&\Psi_{N_m}{\bf\Lambda}_{N_m} \rightarrow
\psi{\bf A}\quad\mbox{strongly in}~~ 
L^2(0,T;({\mathcal  L}^2)^3) ,\\
&\nabla\Psi_{N_m}\cdot{\bf\Lambda}_{N_m} \rightharpoonup
\nabla\psi\cdot{\bf A}\quad
\mbox{weakly in}~~ L^2(0,T;{\mathcal  L}^{6/5}) ,\\
&\Psi_{N_m}|{\bf\Lambda}_{N_m}|^2 \rightarrow
\psi|{\bf A}|^2\quad\mbox{strongly in}~~ 
L^2(0,T;{\mathcal  L}^{3/2}) ,\\
&\Psi_{N_m} \nabla\cdot{\bf\Lambda}_{N_m}\rightarrow \psi\nabla\cdot{\bf A}
\quad\mbox{weakly in}~~ L^2(0,T;{\cal L}^{3/2}),\\
&\Psi_{N_m}^*\bigg(\frac{i}{\kappa} \nabla 
+{\bf \Lambda}_{N_m}\bigg) \Psi_{N_m}
 \rightharpoonup
\psi^*\bigg(\frac{i}{\kappa} \nabla 
+{\bf A}\bigg) \psi\quad\mbox{weakly in}~~ 
L^2(0,T;({\mathcal  L}^{3/2})^3) .
\end{align*}
For any given 
$\varphi\in L^2(0,T;{\cal H}^1)
\hookrightarrow L^2(0,T;{\mathcal  L}^{6})$
and ${\bf a}\in L^2(0,T;{\bf X}_N)
\hookrightarrow L^2(0,T;{\bf L}^{3+\delta})$,
integrating \refe{DVPDE1}-\refe{DVPDE2}
with respect to time, we obtain 
\begin{align*}
&\int_0^T\bigg[\bigg(\eta\frac{\partial \Psi_{N_m}}{\partial t} ,\varphi\bigg)
+ \bigg(\bigg(\frac{i}{\kappa} \nabla + {\bf\Lambda}_{N_m}\bigg)  \Psi_{N_m},
\bigg(\frac{i}{\kappa} \nabla + {\bf\Lambda}_{N_m}\bigg)\varphi\bigg)\bigg]\d t\nn\\
&\qquad\qquad\quad +\int_0^T\bigg[\bigg( (|\Psi_{N_m}|^{2}-1)\Psi_{N_m}
 -i\eta \kappa \Psi_{N_m}\nabla\cdot {\bf\Lambda}_{N_m},\varphi\bigg)\bigg]\d t = 0,
\\[5pt]
&\int_0^T\bigg[\bigg(\frac{\partial {\bf\Lambda}_{N_m}}{\partial t} ,{\bf a}\bigg)
+ \big(\nabla\times{\bf\Lambda}_{N_m},\nabla\times{\bf a}\big)
+\big(\nabla\cdot{\bf\Lambda}_{N_m},\nabla\cdot{\bf a}\big) \bigg]\d t   \nn\\
&=
 \int_0^T\bigg[\big({\bf H}, \nabla\times {\bf a}\big)
 -\bigg( {\rm Re}\bigg[\Psi_{N_m}^*\bigg(\frac{i}{\kappa} \nabla 
+{\bf\Lambda}_{N_m}\bigg)\Psi_{N_m}\bigg],{\bf a}\bigg) \bigg]\d t .
\end{align*}
Letting $m\rightarrow \infty$ in the equations above,
we derive that \refe{VPDE1}-\refe{VPDE2}
hold for any $\varphi\in L^2(0,T;{\cal H}^1)$
and ${\bf a}\in L^2(0,T;{\bf X}_N)$.
Since $L^2(0,T;{\bf X}_N)$ is dense in $L^2(0,T;{\bf H}_{\rm n}({\rm curl,div}))$,
it follows that \refe{VPDE1}-\refe{VPDE2} also hold for any 
$\varphi\in L^2(0,T;{\cal H}^1)$
and ${\bf a}\in L^2(0,T;{\bf H}_{\rm n}({\rm curl,div}))$.

It remains to prove 
$\nabla\times {\bf A}\in L^2(0,T;{\bf H}_{\rm n}({\rm curl,div}))$ and 
$\nabla\cdot{\bf A}\in L^2(0,T;H^1)$. Then 
$(\varphi,{\bf A})$ is a solution of 
\refe{VPDE1}-\refe{VPDE2} in the sense of Definition \ref{DefWSol}.
For this purpose, we consider \refe{PDE2}, which implies
\begin{align*}
&\|\nabla\times(\nabla\times{\bf A})
-\nabla(\nabla\cdot{\bf A})\|_{L^2(0,T;{\bf L}^2)}\\
&\leq C\|\partial_t{\bf A}\|_{L^2(0,T;{\bf L}^2)}
+C\|\psi^*(i\kappa^{-1} \nabla \psi + \mathbf{A} \psi)\|_{L^2(0,T;{\bf L}^2)}
+C\|\nabla\times {\bf H}\|_{L^2(0,T;{\bf L}^2)}\\
&\leq C\|\partial_t{\bf A}\|_{L^2(0,T;{\bf L}^2)}
+C\| \nabla \psi\|_{L^2(0,T;{\mathcal  L}^2)} 
+C\| \mathbf{A} \|_{L^2(0,T;{\bf L}^2)}
+C\|\nabla\times {\bf H}\|_{L^2(0,T;{\bf L}^2)}\\
&\leq C . 
\end{align*}
If we define ${\bf Q}=\nabla\times{\bf A}-{\bf H}$,
then ${\bf Q}$ is a divergence-free vector fields which 
satisfies the equation 
\begin{align*}
\left\{\begin{array}{ll}
\nabla\times(\nabla\times {\bf Q})
= \nabla\times {\bf f}  &\mbox{in}\,\,\,\Omega,\\
{\bf Q}\times{\bf n}=0 &\mbox{on}\,\,\,\partial\Omega, 
\end{array}\right.
\end{align*}
where 
$${\bf f}=\nabla\times{\bf Q}-\nabla(\nabla\cdot{\bf A})
=\nabla\times(\nabla\times{\bf A})
-\nabla(\nabla\cdot{\bf A})-\nabla\times {\bf H}\in L^2(0,T;{\bf L}^2) .$$
The standard energy estimate of the equation above gives 
$\nabla\times {\bf Q}\in L^2(0,T;{\bf L}^2)$,
which implies ${\bf Q}\in L^2(0,T;{\bf H}({\rm curl}))$.
Then we further derive 
$\nabla(\nabla\cdot{\bf A})=\nabla\times {\bf Q} 
-{\bf f} \in L^2(0,T;{\bf L}^2)$.

Existence of a weak solution is proved.

\subsection{Uniqueness of the weak solution}

Suppose that there are two weak solutions 
$(\psi,{\bf A})$ and $(\Psi,{\bf\Lambda})$ for 
the system \refe{PDE1}-\refe{init} 
in the sense of Definition \ref{DefWSol}.
Then we define 
$e=\psi-\Psi$ and ${\bf E}={\bf A}-{\bf\Lambda}$
and consider the difference equations
\begin{align}
&\int_0^T\Big[\big(\eta\partial_t e  ,\varphi\big)
+ \frac{1}{\kappa^2}\big(\nabla  e, \nabla\varphi\big) 
+ \big(|{\bf A}|^2   e,  \varphi\big) \Big]\d t\nn\\
& =\int_0^T\Big[-\frac{i}{\kappa}\big({\bf A}\cdot\nabla e ,\varphi\big)
-\frac{i}{\kappa}\big({\bf E}\cdot\nabla \Psi ,\varphi\big)
+\frac{i}{\kappa}\big(e {\bf A},\nabla\varphi\big)
+\frac{i}{\kappa}\big(\Psi {\bf E},\nabla\varphi\big)  \nn\\
&\quad - \big((|{\bf A}|^2 -|{\bf \Lambda}|^2)  \Psi,  \varphi\big)
 -\big( (|\psi|^{2}-1) \psi-(|\Psi|^{2}-1) \Psi,\varphi\big)\Big]\d t\nn\\
&\quad -\int_0^T\big(i\eta\kappa\psi\nabla\cdot{\bf E}
+i\eta\kappa e\nabla\cdot{\bf \Lambda},\varphi\big)\d t  ,
\label{UErEq1}
\end{align}
and
\begin{align}
&\int_0^T\Big[\big(\partial_t{\bf E} ,{\bf a}\big)
+ \big(\nabla\times{\bf E},\nabla\times{\bf a}\big)
+\big(\nabla\cdot{\bf E},\nabla\cdot{\bf a}\big) \Big]\d t\nn\\
& =-\int_0^T{\rm Re} \bigg( 
\frac{i}{\kappa}( \psi^*\nabla  \psi- \Psi^*\nabla  \Psi)
+  {\bf A}(|\psi|^2-|\Psi|^2)+|\Psi|^2 {\bf E}\, ,\, {\bf a}\bigg) \d t ,
\label{UErEq2}
\end{align}
which hold for any $\varphi\in L^2(0,T;{\mathcal H}^1)$ and 
${\bf a}\in L^2(0,T;{\bf H}_{\rm n}({\rm curl},{\rm div}))$.
Choosing $\varphi(x,t)=e(x,t)1_{(0,t')}(t)$ in \refe{UErEq1} 
and considering the real part,
we obtain
\begin{align*}
& \frac{\eta}{2} \|e(\cdot,t') \|_{{\cal L}^2}^2 
+ \int_0^{t'}\Big(\frac{1}{\kappa^2}\|\nabla  e\|_{{\cal L}^2}^2
+  \|{\bf A}    e\|_{{\bf L}^2}^2\Big)\d t \\
&\leq 
\int_0^{t'}\Big(C\|{\bf A}\|_{{\bf L}^{3+\delta}}\|\nabla e\|_{{\cal L}^2}
 \|e\|_{{\cal L}^{6-4\delta/(1+\delta)}}
+C\|{\bf E}\|_{{\bf L}^{3+\delta}}\|\nabla \Psi\|_{{\cal L}^2}
\|e\|_{{\cal L}^{6-4\delta/(1+\delta)}}\\
&\quad 
+C\| e\|_{{\cal L}^{6-4\delta/(1+\delta)}} \|{\bf A}\|_{{\bf L}^{3+\delta}}
\|\nabla e\|_{{\cal L}^2}  +C\| {\bf E}\|_{{\bf L}^2}\|\nabla e\|_{{\cal L}^2}  \\
&\quad
+C(\|{\bf A}\|_{{\bf L}^{3+\delta}}+\|{\bf \Lambda}\|_{{\bf L}^{3+\delta}})
\|{\bf E}\|_{{\bf L}^2} \|e\|_{{\cal L}^{6-4\delta/(1+\delta)}} +C\| e\|_{{\cal L}^2}^2
+C\|\nabla\cdot{\bf E}\|_{L^2}\|e\|_{{\cal L}^2}\Big)\d t \\
&\leq \int_0^{t'}\Big(C\|\nabla e\|_{L^2}
 (C_\epsilon \|e\|_{{\cal L}^2}+\epsilon\|\nabla e\|_{{\cal L}^2})
 +C\|{\bf E}\|_{{\bf H}_{\rm n}({\rm curl},{\rm div})}
 (C_\epsilon \|e\|_{{\cal L}^2}+\epsilon\|\nabla e\|_{{\cal L}^2})\\
&\quad
+C\|\nabla e\|_{{\cal L}^2}(C_\epsilon \|e\|_{{\cal L}^2}
+\epsilon\|\nabla e\|_{{\cal L}^2}) 
+C\| {\bf E}\|_{{\bf L}^2}\|\nabla e\|_{{\cal L}^2} \\
&\quad
+C\|{\bf E}\|_{{\bf L}^2}(C_\epsilon \|e\|_{{\cal L}^2}
+\epsilon\|\nabla e\|_{{\cal L}^2}) +C\| e\|_{{\cal L}^2}^2
+C\|\nabla\cdot {\bf E}\|_{L^2}\| e\|_{{\cal L}^2} \Big)\d t\\  
&\leq \int_0^{t'}\Big(\epsilon\|\nabla e\|_{{\cal L}^2}^2+
\epsilon\|\nabla\times{\bf E}\|_{{\bf L}^2}^2
+ \epsilon\|\nabla\cdot{\bf E}\|_{L^2}^2 
+C_\epsilon\|e\|_{{\cal L}^2}^2
+ C_\epsilon\|{\bf E}\|_{{\bf L}^2}^2\Big)\d t ,
\end{align*}
where $\epsilon$ can be arbitrarily small.
By choosing ${\bf a}(x,t)={\bf E}(x,t)1_{(0,t')}(t)$ in 
\refe{UErEq2}, we get
\begin{align*}
& \frac{1}{2}\|{\bf E}(\cdot,t')\|_{{\bf L}^2}^2 
+\int_0^{t'}\Big(\|\nabla\times{\bf E} \|_{{\bf L}^2}^2
+\|\nabla\cdot{\bf E} \|_{{\bf L}^2}^2 \Big)\d t\\
&\leq \int_0^{t'}\Big(C \| e\|_{{\cal L}^{6-4\delta/(1+\delta)}}
\|\nabla  \psi\|_{L^2}\| {\bf E}\|_{{\bf L}^{3+\delta}}
+C\|\nabla  e \|_{{\cal L}^2} \| {\bf E}\|_{{\bf L}^2}  \\
&\quad
+ (\|e\|_{L^{6-4\delta/(1+\delta)}}\| {\bf A}\|_{{\bf L}^{3+\delta}}
+\|{\bf E}\|_{{\bf L}^2})\|{\bf E}\|_{L^2}\Big)\d t\\
&\leq \int_0^{t'}\Big(C(C_\epsilon \| e\|_{{\cal L}^2} 
+\epsilon\|\nabla  e \|_{{\cal L}^2})
\| {\bf E}\|_{{\bf H}_{\rm n}({\rm curl},{\rm div})}
+\|\nabla  e \|_{{\cal L}^2} \| {\bf E}\|_{L^2} \\
&\quad + (\|e\|_{{\cal L}^2} +\|\nabla e\|_{{\cal L}^2}
+\|{\bf E}\|_{{\bf L}^2})\|{\bf E}\|_{{\bf L}^2}\Big)\d t\\
&\leq
\int_0^{t'}\Big(\epsilon\|\nabla e\|_{{\cal L}^2}^2
+\epsilon\|\nabla\times{\bf E}\|_{{\bf L}^2}
+\epsilon\|\nabla\cdot{\bf E}\|_{{\bf L}^2}
+C_\epsilon  \|e\|_{{\cal L}^2}^2
+ C_\epsilon  \|{\bf E}\|_{{\bf L}^2}^2\Big)\d t ,
\end{align*}
where $\epsilon$ can be arbitrarily small.
By choosing $\epsilon<\frac{1}{4}
\min(1, \kappa^{-2} )$ and summing up the  
two inequalities above, we have
\begin{align*}
&  \frac{\eta}{2}\|e(\cdot,t')\|_{L^2}^2
+\frac{1}{2}\|{\bf E}(\cdot,t')\|_{L^2}^2 
\leq 
\int_0^{t'}\Big(C\|e\|_{L^2}^2 
+C\| {\bf E}\|_{L^2}^2\Big)\d t ,
\end{align*}
which implies
\begin{align*}
&\max_{t\in(0,T)}
\bigg(\frac{\eta}{2}\|e\|_{L^2}^2
+\frac{1}{2}\|{\bf E}\|_{L^2}^2\bigg)=0  
\end{align*}
via Gronwall's inequality. 
Uniqueness of the weak solution is proved.\medskip
   
The proof of Theorem \ref{MainTHM1} is completed.
\qed

\end{document}